\documentclass{article}%
\usepackage{amssymb,amsmath,xcolor,graphicx,xspace,colortbl, rotating}
\usepackage{amsfonts}
\usepackage{amsmath}
\usepackage{amssymb}%
\usepackage{graphicx}
\providecommand{\U}[1]{\protect \rule{.1in}{.1in}}
\graphicspath{{nestalgKenOct18.pdf_graphics/}{nestalgKenOct18.pdf_tcache/}{nestalgKenOct18.pdf_gcache/}}
\DeclareGraphicsExtensions{.pdf,.eps,.ps,.png,.jpg,.jpeg}
\newtheorem{theorem}{Theorem}

\newtheorem{corollary}[theorem]{Corollary}

\newtheorem{definition}[theorem]{Definition}
\newtheorem{example}[theorem]{Example}

\newtheorem{lemma}[theorem]{Lemma}

\newtheorem{proposition}[theorem]{Proposition}
\newtheorem{remark}[theorem]{Remark}

\newenvironment{proof}[1][Proof]{\noindent \textbf{#1.} }{\  \rule{0.5em}{0.5em}}
\begin{document}

\title{ Nest algebras in an arbitrary vector space}
\author{~Don Hadwin\thanks{ University of New Hampshire, NH., USA } and K.J.
Harrison\thanks{ Murdoch University, WA., Australia } }
\maketitle

\begin{abstract}
We examine the properties of~algebras of linear transformations that leave
invariant all subspaces in a~totally ordered lattice of subspaces of an
arbitrary vector space. We compare our results with those that apply for the
corresponding algebras of bounded operators that act on~a Hilbert space.

\end{abstract}

\section{Introduction \label{SecIntro}}

The study of triangular forms for operators has long been an important part of
the theory of non-self-adjoint operators and operator algebras. See \cite{KRD}
for a detailed account. In \cite{Ring1} Ringrose introduced the terms `nest'
and `nest algebra'. For Ringrose a~nest $\mathfrak{N}$ is a complete, totally
ordered sublattice of the lattice of all closed subspaces of a Hilbert space
$\mathfrak{H}$ that contains the trivial subspaces $\{0\}$ and $\mathfrak{H}$.
The corresponding nest algebra $\operatorname*{Alg} \mathfrak{N}$ is algebra
of all operators on $\mathfrak{H}$ that leave invariant each of the subspaces
in $\mathfrak{N}$.

In this paper we examine~totally ordered lattices of subspaces of an arbitrary
vector space and the associated operator algebras. Here a nest $\mathfrak{N}$
in a vector space $\mathfrak{X}$ is a complete, totally ordered sublattice of
the lattice of all subspaces of $\mathfrak{X}$ that contains the trivial
subspaces $\{0\}$ and $\mathfrak{X}$. The corresponding nest algebra
$\operatorname*{Alg} \mathfrak{N}$ is algebra of all operators on
$\mathfrak{X}$ that leave invariant each of the subspaces in $\mathfrak{N}$.
We obtain results concerning the finite rank operators in~$\operatorname*{Alg}
\mathfrak{N}$ that mirror those that apply in the Hilbert space case. We also
examine the Jacobson radical of~$\operatorname*{Alg} \mathfrak{N}$ and obtain
a simple characterization when the nest satisfies a descending chain
condition. We also show~that the same characterization of the Jacobson radical
holds for other types of nest algebras.

\subsection{ Complete distributivity}

The lattice operations $\wedge$ and $\vee$in~$\mathcal{S}(\mathfrak{X})$,~ the
lattice of all subspaces of the vector space $\mathfrak{X}$, are
intersection~and linear span. In particular, if~$\mathcal{M}$ and~$\mathcal{N}%
$ are subspaces of~$\mathfrak{X}$,~$\mathcal{M} \vee \mathcal{N}
=\operatorname*{span} \{ \mathcal{M} ,\mathcal{N}\} =\{x +y :x \in \mathcal{M}
,y \in \mathcal{N}\}$. However in a totally ordered sublattice the lattice
operations are simply the set operations~ $\cap$ and~$\cup$. So any nest
$\mathfrak{N}$ is completely distributive (see \cite{KRD}).

Suppose that $\mathfrak{N}$ is a nest in $\mathfrak{X}$. For each $x
\in \mathfrak{X}$ we define%
\begin{equation}
\mathfrak{N}(x) =\bigcap \{ \mathcal{M} \in \mathfrak{N} :x \in \mathcal{M}%
\} \  \text{and}\  \mathfrak{N}(x)_{ -} =\bigcup \{ \mathcal{M} \in \mathfrak{N} :x
\notin \mathfrak{N}\} .\label{atom}%
\end{equation}
It follows easily from (\ref{atom}) that
\begin{equation}
x \in \mathcal{N} \iff \mathfrak{N}(x) \subseteq \mathcal{N}\  \text{and}\ x
\notin \mathcal{N} \iff \mathcal{N} \subseteq \mathfrak{N}(x)_{ -}%
\label{granular}%
\end{equation}

\begin{lemma}
The join-irreducible elements of the completely distributive lattice
$\mathfrak{N}$ are the subspaces of the form $\mathfrak{N}(x)$ where $x$ is
any non-zero vector in $\mathfrak{X}$.
\end{lemma}

\begin{proof}
Suppose that $x \neq0$, and that $\mathfrak{N}(x) \subseteq \bigcup \{N :N
\in \mathfrak{N}^{\#}\}$ where $\mathfrak{N}^{\#} \subseteq \mathfrak{N}$. Then
$x \in \bigcup \{N :N \in \mathfrak{N}^{\#}\}$ by (\ref{granular}). So $x \in N$
for some~$N \in \mathfrak{N}^{\#}$, and it follows from (\ref{granular}) that
$\mathfrak{N}(x) \subseteq N$. So $\mathfrak{N}(x)$ is join-irreducible.

Suppose now that~$N$ is a join-irreducible subspace in $\mathfrak{N}$.
Clearly~$\mathcal{N} =\bigcup \{ \mathfrak{N}(x) :x \in \mathcal{N}\} ,$ and
so~$\mathcal{N} \subseteq \mathfrak{N}(x)$ for some~$x \in \mathcal{N} .$
So~$\mathcal{N} =\mathfrak{N}(x)$.
\end{proof}

\begin{remark}
Complete distributivity distinguishes the vector space case from the Hilbert
space case. Some of~the most interesting nests of closed subspaces of a
Hilbert space are `continuous', have no join-irreducible elements, and are not
completely distributuve.~
\end{remark}

\section{Finite rank operators}

The \emph{rank} of an operator in~$\mathcal{L}(\mathfrak{X})$ is~the dimension
of its range.\ In this section we examine the properties of~operators in a
nest algebra $\mathcal{A} =\operatorname*{Alg} \mathfrak{N}$ whose ranks are
finite. Let $\mathcal{R}$ denote the set of finite-rank operators
in~$\mathcal{L}(\mathfrak{X})$. Various authors have investigated the
properties of $\mathcal{R} \cap \mathcal{A}$~in the Hilbert space context. For
example, Erdos proved \cite{Erd1} that if~$\mathcal{N}$ is a nest of closed
subspaces of a Hilbert space then the strong closure of $\mathcal{R}
\cap \mathcal{A}$ is $\mathcal{A}$.

Rank-one operators also have an important role in the Hilbert space context.
Suppose that $T \in \mathcal{R}_{1}$, where~$\mathcal{R}_{1}$ denotes the set
of all rank-one operators in~$\mathcal{L} (\mathfrak{X})$. Then there exists
$y \in \mathfrak{X}$ such that for all $x \in \mathfrak{X}$, $T x =\varphi(x) y$
where $\varphi(x) \in \mathbb{F}$. Since $T$ is linear the map $x
\rightarrow \varphi(x)$ is a linear functional of $\mathfrak{X}$. Let
$\mathfrak{X}^{ \prime}$ denote the algebraic dual of $\mathfrak{X}$, i.e. the
set of all linear maps from $\mathfrak{X}$ into $\mathbb{F}$. Each~rank-one
operator on $\mathfrak{X}$ has the form $x \otimes \varphi$, where $x
\in \mathfrak{X}$, $\varphi \in \mathfrak{X}^{ \prime}$, and $(x \otimes \varphi)
(y) =\varphi(y) x$ for all~$y \in \mathfrak{X}$.

The following lemma characterizes the rank-one operators in~$\mathcal{A} .$

\begin{lemma}
\label{rank1inA}Suppose that~$x \in \mathcal{X}$ and $\varphi \in \mathcal{X}^{
\prime} .$ Then~$x \otimes \varphi \in \mathcal{R}_{1} \cap \mathcal{A}$ if and
only if~$\mathfrak{N}_{ -}(x) \subseteq \ker \varphi.$

\begin{proof}
First suppose that~$x \otimes \varphi \in \mathcal{R}_{1} \cap \mathcal{A} ,$ and
that $y \in \mathfrak{N}_{ -}(x)$. Since~$\mathfrak{N}_{ -}(x) \in \mathcal{A}
,$ $(x \otimes \varphi)(y) =\varphi(y)x \in \mathfrak{N}_{ -}(x)$. Since $x
\notin$$\mathfrak{N}_{ -}(x)$ it follows that $\varphi(y) =0$.
So~$\mathfrak{N}_{ -}(x) \subseteq \ker \varphi$.

Now suppose that~$\mathfrak{N}_{ -}(x) \subseteq \ker \varphi$ and that~$N
\in \mathfrak{N}$. If~$N \subset \mathfrak{N}(x)$ then~$N \subseteq \mathcal{N}_{
-}(x)$ and $(x \otimes \varphi)N =\{0\} \subseteq N$. If~$\mathfrak{N}(x)
\subseteq N$ then $(x \otimes \varphi)N =\operatorname*{span} x \subseteq
\mathfrak{N}(x) \subseteq N .$ So $x \otimes \varphi \in \mathcal{R}_{1}
\cap \mathcal{A} .$
\end{proof}
\end{lemma}

\subsection{Reflexivity of $\mathfrak{N}$}

For any subset of~$\mathfrak{A}$ of $\mathcal{L}(\mathfrak{X})$ let
$\operatorname*{Lat} $$\mathfrak{A}$ denote the sublattice of $\mathcal{S}%
(\mathfrak{X})$ consisting of all subspaces of $\mathfrak{X}$ that are
invariant under each of the operators in $\mathfrak{A}$. We shall show that
\begin{equation}
\mathfrak{N} =\operatorname*{Lat} (\mathcal{R}_{1} \cap \mathcal{A})
,\label{rank1ref}%
\end{equation}
from which it follows that $\mathfrak{N}$ is reflexive, i.e. $\mathfrak{N}
=\operatorname*{Lat} \operatorname*{Alg} \mathfrak{N} .$

Longstaff shows in (\cite{Long1}) that (\ref{rank1ref}) holds in the Hilbert
space context.

The following lemma will be used to establish the reflexivity of
$\mathfrak{N}$.

\begin{lemma}
\label{onto}If $x$ and $y$ are non-zero vectors in~$\mathfrak{X}$ and~$y
\in \mathfrak{N}(x)$, then there exists $R \in \mathcal{R}_{1} \cap \mathcal{A}$
such that $Rx =y .$

\begin{proof}
Since $y \in \mathfrak{N}(x) ,$~$\mathfrak{N}(y)_{ -} \subset \mathfrak{N}(y)
\subseteq \mathfrak{N}(x) .$ So~$x \notin \mathfrak{N}(y)_{ -}$, and hence~there
exists $\varphi \in X^{ \prime}$ such that $\varphi(x) =1$ and~$\mathfrak{N}%
(y)_{\_} \subseteq \ker \varphi.$ Then $R =y \otimes \varphi \in \mathcal{R}_{1}
\cap \mathcal{A}$ and~$Rx =\varphi(x)y .$
\end{proof}
\end{lemma}

\begin{theorem}
\label{reflexlem}$\mathfrak{N}$ is reflexive.

\begin{proof}
We shall show that $\mathfrak{N} =$$\operatorname*{Lat} (\mathcal{R}_{1}
\cap \mathcal{A})$. Clearly $\mathfrak{N} \subseteq \operatorname*{Lat}
(\mathcal{R}_{1} \cap \mathcal{A})$.~Suppose that~$N \in \operatorname*{Lat}
(\mathcal{R}_{1} \cap \mathcal{A})$. It is enough to show that~$N
\in \mathfrak{N}$.

Suppose that $x$ and $y$ are non-zero vectors in~$N$ and~$\mathfrak{N}(x)$
respectively. So by Lemma \ref{onto} there exists $R \in \mathcal{R}_{1}
\cap \mathcal{A}$ such that $Rx =y .$ Since~$N \in \operatorname*{Lat}
(\mathcal{R}_{1} \cap \mathcal{A})$, it follows that~$y \in N$, and
hence~$\mathfrak{N}(x) \subseteq N$.

Clearly~$N \subseteq \bigcup \{ \mathfrak{N}(x) :x \in N\}$, and so%
\[
N \subseteq \bigcup \{ \mathfrak{N}(x) :x \in N\} \subseteq N
\]
So~$N =\bigcup \{ \mathfrak{N}(x) :x \in N\} \in \mathfrak{N}$, as required.
\end{proof}
\end{theorem}

\subsection{ Finite rank idempotents}

A simple calculation shows that $(x_{1} \otimes \varphi_{1})(x_{2}
\otimes \varphi_{2}) =\varphi_{1}(x_{2})(x_{1} \otimes \varphi_{2})$. So $x
\otimes \varphi$ is idempotent if and only if~$\varphi(x) =1$.

The following lemma concerning rank-one idempotents in $\mathcal{A}$ will be useful.

\begin{lemma}
\label{idemdecomp}Suppose that~$M$ is a finite-dimensional subspace
of~$\mathfrak{X}$. Then~$M =\operatorname*{ran} P$ for some idempotent $P
\in \mathcal{A}$. Furthermore,~$P$ is the sum of~$n$ rank-one idempotents in
$\mathcal{A}$, where $n =\dim M$.

\begin{proof}
The proof is by induction on $\dim M$. First suppose that~$\dim M =1$, and
choose a non-zero vector $x \in M$. Now choose~$\varphi \in \mathfrak{X}^{
\prime}$ such that~$\mathfrak{N}(x)_{ -} \subseteq \ker \varphi$ and~$\varphi(x)
=1$. Such a $\varphi$ exists because $x \notin \mathfrak{N}(x)_{ -}$.Then~$x
\otimes \varphi$ is the required idempotent.

Now suppose that~$n =\dim M >1$ and that the result is true for all subspaces
of~$\mathfrak{X}$ with dimension less than~$n$. Choose a non-zero vector~$y
\in M$ and a subspace~$M^{\#}$ of~$M$ such that~$M^{\#}$ and
$\operatorname*{span} y$ are complementary subspaces of~$M$, i.e.~$M =M^{\#}
+\operatorname*{span} y =M$~ and~$M^{\#} \cap \operatorname*{span} y =\{0\}$.
By the induction hypothesis there exists an idempotent $P^{\#} \in \mathcal{A}$
such that $\operatorname*{ran} P^{\#} =M^{\#}$, and rank-one
idempotents~$P_{1} ,P_{2} , \cdots,P_{n -1}$ in~$\mathcal{A}$ such
that~$P^{\#} =P_{1} +P_{2} + \cdots+P_{n -1}$. Let~$x =y -P^{\#}y$. Then $0
\neq x \in M$ and $P^{\#}x =0$. Suppose that $x =u +v$, where $u
\in \mathfrak{N}(x)_{ -}$ and~$v \in M^{\#}$.~ Then~$P^{\#}x =P^{\#}u +P^{\#}%
v$, i.e.~$0 =P^{\#}u +v$, since~$M^{\#} =\operatorname*{ran} P^{\#}$
and~$P^{\#}$ is idempotent. So~$x =u -P^{\#}u$. Since~$P^{\#} \in \mathcal{A}$
it follows that $u -P^{\#}u \in \mathcal{N}(x)_{ -}$. Since~$x \notin
\mathcal{N}(x)_{ -}$ we have~a contradiction. So~$x \notin \mathfrak{N}(x)_{ -}
+M^{\#} =\mathfrak{N}(x)_{ -} +\operatorname*{ran} P^{\#}$, and hence there
exists~$\varphi \in \mathfrak{X}^{ \prime}$ such that~%
\[
\varphi(x) =1 ,\  \  \text{and}\  \  \mathfrak{N}(x)_{ -} +\operatorname*{ran}
P^{\#} \subseteq \ker \varphi \  \  \
\]

Let $P_{n} =x \otimes \varphi$. Then $P_{n}$ is idempotent since~$\varphi(x)
=1$, and $P_{n} \in \mathcal{A}$ since~$\mathfrak{N}(x)_{ -} \subseteq
\ker \varphi$. Furthermore,~ $P^{\#}P_{n} =P^{\#}x \otimes \varphi=0 ,$ and
$P_{n}P^{\#} =x \otimes \varphi P^{\#} =0$ since $\operatorname*{ran} P^{\#}
\subseteq \ker \varphi.$ Now let~$P =P^{\#} +P_{n}$. Then
\[
P^{2} =(P^{\#})^{2} +P^{\#}P_{n} +P_{n}P^{\#} +P_{n}^{2} =P^{\#} +P_{n} =P ,
\]
and $\operatorname*{ran} P =\operatorname*{ran} P^{\#} +\operatorname*{ran}
P_{n} =M^{\#} +\operatorname*{span} x =M$, as required.
\end{proof}
\end{lemma}

\subsection{Rank decomposition}

Lemma \ref{idemdecomp} provides an easy proof of a rank-decomposition property
of finite rank operators in the nest algebra $\mathcal{A} .$

\begin{theorem}
\label{rnkdecomp}Suppose that $T$ is a finite rank operator in $\mathcal{A}$.
Then $T$ is the sum of $n$ rank-one operators in $\mathcal{A}$, where $n
=\operatorname*{rank} T .$

\begin{proof}
By Lemma \ref{idemdecomp}, $\operatorname*{ran} T =\operatorname*{ran} P$ for
some idempotent $P$ in $\mathcal{A}$.~ Furthermore $P =P_{1} +P_{2} +
\cdots+P_{n}$ where each $P_{k}$ is a rank-one idempotent in $\mathcal{A}$.
Let $T_{k} =P_{k}T$ for $1 \leq k \leq n$. Then $T_{k} \in \mathcal{A}$ and
$\operatorname*{rank} T_{k} \leq1$ for each $k$. Furthermore,
\[
T =PT =\sum_{k =1}^{n}P_{k}T =\sum_{k =1}^{n}T_{k} .
\]
This is the required decomposition. The Hilbert version of this result was
proved by Ringrose (\cite{Erd1}).
\end{proof}
\end{theorem}

\begin{remark}
The proof of Theorem \ref{rnkdecomp} is easily modified to show that if $T$ is
a finite rank operator in $\mathcal{I}$, where $\mathcal{I}$ is a left ideal
in $\mathcal{A}$, then $T$ is the sum of $n$ rank-one operators in
$\mathcal{I}$, where $n =\operatorname*{rank} T .$
\end{remark}

\subsection{Density}

Lemma \ref{idemdecomp} also provides an easy proof of a density property of
the linear span of rank-one operators in $\mathcal{A} .$ First we introduce a
special topology on $\mathcal{L}(\mathfrak{X})$.

\begin{definition}
The set of all subsets of $\mathcal{L}(\mathfrak{X})$ of the form
\[
\mathcal{U}(T ,x) =\{S \in \mathcal{L}(\mathfrak{X}) :Sx =Tx\  \} ,\text{ }%
\]
where $x \in \mathfrak{X}$~and $T \in \mathcal{L}(\mathfrak{X})$, is a set of
subbasic neighbourhoods of $T$ for~the strict topology on $\mathcal{L}%
(\mathfrak{X})$
\end{definition}

\begin{theorem}
\label{strictd}The~span of the rank-one operators in $\mathcal{A}$ is strictly
dense in $\mathcal{A} .$

\begin{proof}
Suppose that $T \in$$\mathcal{A}$ and that~$\mathcal{F}$ is a finite
subset~of~$\mathfrak{X}$. Let $\mathcal{R}_{1}^{\#} \cap \mathcal{A}$ denote
the span of~$\mathcal{R}_{1} \cap \mathcal{A}$. We need to show that there
exists $S \in \mathcal{R}_{1}^{\#} \cap \mathcal{A}$ such that~$Sx =Tx$ for all
$x \in \mathcal{F}$.

By Lemma \ref{idemdecomp} $\operatorname*{span} \mathcal{F}
=\operatorname*{ran} P$ for some idempotent $P \in \mathcal{A}$.
Furthermore,~$P$ is the sum of~$n$ rank-one idempotents in $\mathcal{A}$,
where $n =\dim \operatorname*{span} \mathcal{F}$. Let $T_{k} =TP_{k}$ for $1
\leq k \leq n$. Then $T_{k} \in \mathcal{A}$ and $\operatorname*{rank} T_{k}
\leq1$ for each $k$. So~$S =\sum_{k =1}^{n}T_{k} \in \mathcal{R}_{1}^{\#}
\cap \mathcal{A}$. Furthermore, for each $x \in \operatorname*{span}
\mathcal{F}$,
\[
Tx =TPx =\sum_{k =1}^{n}TP_{k}x =Sx .
\]
as required.
\end{proof}
\end{theorem}

\begin{remark}
The proof of Theorem~\ref{strictd} is easily modified to show that if $T$ is a
finite rank operator in $\mathcal{I}$, where $\mathcal{I}$ is a right ideal in
$\mathcal{A}$, then $T$ is the sum of $n$ rank-one operators in $\mathcal{I}$,
where $n =\operatorname*{rank} T .$
\end{remark}

\section{Dual nests}

For any subset~$M$ of $\mathfrak{X}$, let~$M^{\perp}$ denote the annihilator
of~$M$, i.e.~%
\[
M^{ \bot} =\{ \varphi \ :\  \varphi \in \mathfrak{X}^{ \prime}\text{}%
\text{}\  \text{and}\  \ M \subseteq \ker \varphi \} \text{ }%
\]
Suppose that $\mathfrak{N}$ is a nest of subspaces of~$\mathfrak{X} ,$ and
that~$\mathfrak{N}^{ \bot} =\{M^{ \bot} :M \in \mathfrak{N}\}$. We
call~$\mathfrak{N}^{ \bot}$ the dual of the nest $\mathfrak{N}$. Since the
map~$M \mapsto M^{ \bot}$ is order reversing, i.e.~$M_{1} \subseteq M_{2}
\Longleftrightarrow M_{1}^{ \bot} \supseteq M_{2}^{ \bot}$,~$\mathfrak{N}^{
\bot}$ is a linearly ordered family~of subspaces of $\mathfrak{X}^{ \prime}$
that is anti-order isomorphic to $\mathfrak{N}$.

We are interested in the issue of completeness of $\mathfrak{N}^{ \bot}$.

\begin{lemma}
\label{dualcomplete}For any family $\{M_{\alpha} :\alpha \in \Psi \}$ of
subspaces in $\mathfrak{N}$,
\[
\bigcap_{\alpha \in \Psi}M_{\alpha}^{ \bot} =\left( \bigcup_{\alpha \in \Psi
}M_{\alpha}\right) ^{ \bot}\  \text{and}\  \bigcup_{\alpha \in \Psi}M_{\alpha}^{
\bot} \subseteq \left( \bigcap_{\alpha \in \Psi}M_{\alpha}\right) ^{ \bot}%
\]

\begin{proof}
Suppose that $\varphi \in \mathfrak{X}^{ \prime}$. It is easy to see that
\[
\varphi \in \bigcap_{\alpha \in \Psi}M_{\alpha}^{ \bot} \iff M_{\alpha}
\subseteq \ker \varphi \  \text{for all}\  \alpha \in \Psi \iff \varphi \in \left(
\bigcup_{\alpha \in \Psi}M_{\alpha}\right) ^{\perp} .
\]

Similarly, if $\varphi \in \bigcup_{\alpha \in \Psi}M_{\alpha}^{ \bot}$
then~$M_{\alpha^{\#}} \subseteq \ker \varphi$ for some~$\alpha^{\#} \in \Psi$. It
follows that~$\bigcap_{\alpha \in \Psi}M_{\alpha} \subseteq \ker \varphi$,
i.e.~$\varphi \in \left( \bigcap_{\alpha \in \Psi}M_{\alpha}\right) ^{ \bot} .$
\end{proof}
\end{lemma}

\begin{corollary}
$\mathfrak{N}^{\perp}$ is complete if and only if $\bigcup_{\alpha \in \Psi
}M_{\alpha}^{ \bot} =\left( \bigcap_{\alpha \in \Psi}M_{\alpha}\right) ^{ \bot}$
for each family~$\{M_{\alpha} :\alpha \in \Psi \}$ of subspaces in $\mathfrak{N}$.

\begin{proof}
In the light of Lemma~\ref{dualcomplete} it is sufficient to show that
if~$\mathfrak{N}^{ \bot}$ is complete and $\{M_{\alpha} :\alpha \in \Psi \}$ is a
family of subspaces in $\mathfrak{N}$, then~$\left( \bigcap_{\alpha \in \Psi
}M_{\alpha}\right) ^{ \bot} \subseteq \bigcup_{\alpha \in \Psi}M_{\alpha}^{\perp}
.$

If~$\mathfrak{N}^{ \bot}$ is complete, $\bigcup_{\alpha \in \Psi}M_{\alpha
}^{\perp} =M_{\#}^{\perp}$ for some~$M_{\#} \in \mathfrak{N}$. Suppose
that~$\alpha_{0} \in \Psi$. Then~$M_{\alpha_{0}}^{\perp} \subseteq
\bigcup_{\alpha \in \Psi}M_{\alpha}^{\perp} =M_{\#}^{\perp} ,$ and so~$M_{\#}
\subseteq M_{\alpha_{0}}$$.$ Therefore~$M_{\#} \subseteq \bigcap_{\alpha \in
\Psi}M_{\alpha} ,$ and so $\left( \bigcap_{\alpha \in \Psi}M_{\alpha}\right) ^{
\bot} \subseteq M_{\#}^{\perp} =\bigcup_{\alpha \in \Psi}M_{\alpha}^{\perp} ,$
as required.
\end{proof}
\end{corollary}

\begin{example}
\label{c00(N)}Suppose that $\mathfrak{X} =c_{00}(\mathbb{N})$, the vector
space of all finitely non-zero~$\mathbb{F}$-valued sequences.~Then
$\mathfrak{X}^{ \prime}$ can be regarded as the vector space of all~
$\mathbb{F}$-valued sequences. If~$f =(f(k))_{k =1}^{\infty} \in \mathfrak{X}$
and $\varphi=(\varphi(k))_{k =1}^{\infty} \in \mathfrak{X}^{ \prime}$,~then
$\varphi(f) =\sum_{k =1}^{\infty}\varphi(k)f(k) .$

For each $n \in \mathbb{N}$, let~$M_{n} =\{f \in \mathfrak{X}
:\operatorname*{supp} f \subseteq \{1 ,2 ,3 , \cdots,n\} \}$, where
$\operatorname*{supp} (f(k))_{k =1}^{\infty} =\{k :f(k) \neq0\}$ , and let
\[
\mathfrak{N} =\{ \{0\} ,M_{1} ,M_{2} ,M_{3} , \cdots,\mathfrak{X}\} .
\]
Then~$\mathfrak{N}$ is a complete, totally ordered family of subspaces of
$\mathfrak{X} ,$ i.e.~$\mathfrak{N}$ is a nest.

Note that~$M_{n}^{ \bot} =\{ \varphi \in \mathfrak{X}^{ \prime}%
\ :\operatorname*{supp} \varphi \subseteq \{n +1 ,n +2 ,n +3 , \cdots \} .$ It is
easy to see that $\mathfrak{N}^{ \bot} =\{ \mathfrak{X}^{ \prime} ,M_{1}%
^{\perp} ,M_{2}^{\perp} ,M_{3}^{\perp} , \cdots,\{0\} \}$ is a complete,
totally ordered family of subspaces of~$\mathfrak{X}^{ \prime} ,$
i.e.~$\mathfrak{N}^{ \bot}$ is a nest.~
\end{example}

\begin{example}
\label{c00(N)1}Suppose that $\mathfrak{X} =c_{00}(\mathbb{N})$ as in Example
\ref{c00(N)}, and let
\[
\mathfrak{N}^{\#} =\{ \mathfrak{X} ,M_{1}^{\#} ,M_{2}^{\#} ,M_{3}^{\#} ,
\cdots,\{0\} \} ,
\]
where~$M_{n}^{\#} =\{f \in \mathfrak{X} :\operatorname*{supp} f \subseteq \{n +1
,n +2 ,n +3 , \cdots \}$~for each $n \in \mathbb{N}$. Then~$\mathfrak{N}^{\#}$
is a complete, totally ordered family of subspaces of $\mathfrak{X} ,$ i.e.
$\mathfrak{N}^{\#}$ is a nest.

Note that $(M_{n}^{\#})^{ \bot} =M_{n} =$$\{ \varphi \in \mathfrak{X}^{ \prime}
:\operatorname*{supp} \varphi \subseteq \{1 ,2 ,3 , \cdots,n\} \}$ as in Example
\ref{c00(N)}. So $(M_{1}^{\#})^{ \bot} ,(M_{2}^{\#})^{ \bot} ,(M_{3}^{\#})^{
\bot} , \cdots$ is a strictly increasing sequence in $(\mathfrak{N}^{\#})^{
\bot}$, and $\bigcup_{n =1}^{\infty}(M_{n}^{\#})^{ \bot}$$=\mathfrak{X}
\notin \left( \mathfrak{N}^{\#}\right) ^{ \bot} .$ So~$(\mathfrak{N}^{\#})^{
\bot}$ is not complete.
\end{example}

The nest~$\mathfrak{N}^{\#}$ in Example~\ref{c00(N)1} has a strictly
decreasing, infinite sequence of subspaces, i.e.,~it is not well-ordered.~The
following lemma shows that this is the key to~the incompleteness of
$(\mathfrak{N}^{\#})^{ \bot}$.

\begin{lemma}
\label{wellorddual}Suppose that $\mathfrak{N}$ is a complete nest of subspaces
of a vector space $\mathfrak{X}$. Then $\mathfrak{N}^{ \bot}$ is complete if
and only if~$\mathfrak{N}$ is well-ordered.

\begin{proof}
First suppose that $\mathfrak{N}$ is well-ordered, and that $\{M_{\alpha}
:\alpha \in \Psi \}$ is a family of subspaces in $\mathfrak{N}$. In the light of
Corollary \ref{dualcorr} it is sufficient to show that~ $\  \left(
\bigcap_{\alpha \in \Psi}M_{\alpha}\right) ^{ \bot}\  \subseteq \bigcup_{\alpha
\in \Psi}M_{\alpha}^{ \bot} .$

Since $\mathfrak{N}$ is well-ordered, $\cap_{\alpha \in \Psi}M_{\alpha}
=M_{\alpha^{\#}}$ for some $\alpha^{\#} \in \Psi$.~So%
\[
\left( \bigcap_{\alpha \in \Psi}M_{\alpha}\right) ^{ \bot} =M_{\alpha^{\#}}^{
\bot} \subseteq \bigcup_{\alpha \in \Psi}M_{\alpha}^{ \bot} ,\  \text{as required.
}%
\]

Now suppose that $\mathfrak{N}$ is not well-ordered, and that~$M_{1} ,M_{2}
,M_{3} , \cdots$ is a strictly decreasing infinite sequence of subspaces
in~$\mathfrak{N}$. For each $n \in \mathbb{N}$ choose $x_{n}$ such that~$x_{n}
=M_{n} \setminus M_{n +1}$. Then $\{x_{1} ,x_{2} ,x_{3} , \cdots \}$ is a
linearly independent set and $\operatorname*{span} \{x_{1} ,x_{2} ,x_{3} ,
\cdots \} \cap M_{\infty} =\{0\}$, where~$M_{\infty} = \cap_{n =1}^{\infty
}M_{n}$. So there exists $\varphi \in \mathfrak{X}^{ \prime}$ such that~%
\begin{equation}
\varphi(x_{n}) =1\  \text{for each}\ n \in \text{}\mathbb{N}\  \  \text{and}%
\  \ M_{\infty} \subseteq \ker \varphi \label{fnal1}%
\end{equation}
It follows easily from (\ref{fnal1}) that $\varphi \in M_{\infty}^{ \bot}
\setminus \left( \bigcup_{n =1}^{\infty}M_{n}^{ \bot}\right)  .$ So~%
\begin{equation}
\bigcup_{n =1}^{\infty}M_{n}^{ \bot} \subset M_{\infty}^{ \bot}\label{fnal2}%
\end{equation}

Suppose that $\bigcup_{n =1}^{\infty}M_{n}^{ \bot} \in \mathfrak{N}^{\perp}$,
i.e. $\bigcup_{n =1}^{\infty}M_{n}^{ \bot} =M^{\perp}$ for some~$M
\in \mathfrak{N}$. Then~$M_{n}^{\perp} \subseteq M^{\perp}$ and~$M \subseteq
M_{n}$ for each $n \in \mathbb{N}$. So~$M \subseteq M_{\infty} ,$ and
hence~$M_{\infty}^{ \bot} \subseteq M^{\perp}$. But this contradicts
(\ref{fnal2}), and so there is no such subspace~$M$ in~$\mathfrak{N} .$ So
$\mathfrak{N}^{\perp}$ is not complete.
\end{proof}
\end{lemma}

\section{The Jacobson radical}

Suppose that $\mathcal{R}$ is a ring with identity $1$. The Jacobson radical
$\operatorname*{Rad} \mathcal{R}$ is the intersection of all maximal left
ideals of $\mathcal{R}$. It is also the intersection of all maximal right
ideals of $\mathcal{R}$.~ See (\cite{Is}). A more useful characterisation
of~$\operatorname*{Rad} \mathcal{R}$ is the following:

\begin{proposition}
\label{Raddef}Suppose that $T \in \mathcal{R}$. The following are equivalent:

1.\qquad$T \in \operatorname*{Rad} \mathcal{R}$

2.\qquad$1 -AT$ is invertible in~$\mathcal{R}$ for each~$A \in \mathcal{A}$

3.\qquad$1 -TA$ is invertible in $\mathcal{R}$ for each $A \in \mathcal{A}$
\end{proposition}

\begin{definition}
Suppose that~$\mathfrak{N}$ is a nest on $\mathfrak{X}$ and that
$\mathcal{A}=\operatorname*{Alg}\mathfrak{N}$. The strictly triangular~ideal
$\mathcal{A}_{-}$ is defined by
\[
\mathcal{A}_{-}=\{T:T\in \mathcal{A}\  \text{and}\ Tx\in \mathfrak{N}%
(x)_{-}\  \text{for all}\ x\in \mathfrak{X}\}
\]

\end{definition}

\begin{lemma}
\label{Rad1}Suppose that~$\mathfrak{N}$ is a nest on $\mathfrak{X}$ and that
$\mathcal{A} =\operatorname*{Alg} \mathfrak{N}$. Then
\[
\operatorname*{Rad} \mathcal{A} \subseteq \mathcal{A}_{ -} .
\]

\begin{proof}
Suppose that $T \in \mathcal{A} \setminus \mathcal{A}_{ -} .$ Then $Tx
\notin \mathfrak{N}(x)_{ -}$ for some~$x \in \mathfrak{X}$.~Choose $\varphi \in
X^{ \prime}$ such that $\varphi(Tx) =1$ and~$\mathfrak{N}(x)_{ -}
\subseteq \ker \varphi$. It follows from (\ref{rank1inA}) that~$x \otimes
\varphi \in \mathcal{A}$.

Now~$(1 -(x \otimes \varphi)T)x =x -\varphi(Tx)x =0$. So~$1 -(x \otimes
\varphi)T$ is not invertible and so $T \notin \operatorname*{Rad} \mathcal{A}$
by Proposition \ref{Raddef}.
\end{proof}
\end{lemma}

We now seek conditions which are either necessary or~sufficient for the
equality of the radical~$\operatorname*{Rad} \mathcal{A}$ and the strictly
triangular~ideal $\mathcal{A}_{ -}$ The notion of local nilpotence will be useful.

\begin{definition}
We say that $T \in \mathcal{L} (\mathfrak{X})$ is nilpotent at $x
\in \mathfrak{X}$ if $T^{n} x =0$ for sufficiently large $n$. We say that $T$
is locally nilpotent if it is nilpotent at each $x \in \mathfrak{X}$.
\end{definition}

\begin{lemma}
If~each $T \in \mathcal{A}_{ -}$ is locally nilpotent,
then~$\operatorname*{Rad} \mathcal{A} =\mathcal{A}_{ -}$.

\begin{proof}
Suppose that $T \in \mathcal{A}_{ -}$~and that $A \in \mathcal{A}$. Then $AT
\in \mathcal{A}_{ -}$ and hence is locally nilpotent by assumption.

~ Let $S =$~ $1 +\sum_{n =1}^{\infty}(AT)^{n}$. The sum $S$ is well-defined as
an operator in $\mathcal{L}(\mathfrak{X})$, because the local nilpotence of
$AT$ ensures that for each $x \in \mathfrak{X}$ the series $\sum_{n =1}%
^{\infty}(AT)^{n}x$ has only finitely many non-zero terms. If $x \in M$ for
some~$M \in \mathfrak{N}$, it is clear that $Sx \in M$. So $S \in \mathcal{A}$.
Furthermore, it is easy to see that $S(1 -AT) =(1 -AT)S =1$. So $S$ is the
inverse of $1 -AT$ in $\mathcal{A}$, and hence $T \in \operatorname*{Rad}
\mathcal{A}$.
\end{proof}
\end{lemma}

\begin{lemma}
\label{wellordlocnil}If~$\mathfrak{N}$ is well-ordered~then each $T
\in \mathcal{A}_{ -}$ is locally nilpotent.

\begin{proof}
Suppose that $T \in \mathcal{A}_{ -}$ is not locally nilpotent. Then there
exists $x \in \mathfrak{X}$ such that $T^{n}x \neq0$ for all $n \in \mathbb{N}
.$~Since $T \in \mathcal{A}_{ -}$, for each~$n \in \mathbb{N} ,$%
\[
\mathfrak{N}(T^{n +1}x) \subseteq T(\mathfrak{N}(T^{n}x)) \subseteq
(\mathfrak{N}(T^{n}x))_{ -} \subset \mathfrak{N}(T^{n}x)
\]
So~$\mathfrak{N}(T^{n}x) :n =1 ,2 ,3 , \cdots$ is a strictly decreasing,
infinite sequence of subspaces in~$\mathfrak{N}$, and~hence~$\mathfrak{N}$ is
not well-ordered.
\end{proof}
\end{lemma}

\begin{corollary}
\label{wellordcor}If~$\mathfrak{N}$ is well-ordered~then $\operatorname*{Rad}
\mathcal{A} =\mathcal{A}_{ -} .$
\end{corollary}

The following result shows that for dual nests,~well-ordering is not essential
for the equality of the radical and the strictly triangular~ideal.

\begin{theorem}
\label{Donthmmod}Suppose that $\mathfrak{N}$ is a~nest of subspaces of a
vector space $\mathfrak{X}$ whose order type is~$\omega$, the first infinite
ordinal. Then~$\mathfrak{N}^{ \bot}$ is a~nest of subspaces of~$\mathfrak{X}^{
\prime}$, whose order type is anti-isomorphic to~$\omega$, and
$(\operatorname*{Alg} \mathfrak{N}^{ \bot})_{ -} =\operatorname*{Rad}
(\operatorname*{Alg} \mathfrak{N}^{ \bot}) .$

\begin{proof}
In view of Lemma \ref{wellorddual} it is sufficient to show that
$\mathcal{A}_{ -} =\operatorname*{Rad} \mathcal{A}$, where $\mathcal{A}
=\operatorname*{Alg} \mathfrak{N}^{ \bot} .$~

Let~$M_{0} =\{0\}$, and for each~$n >0$ let~$M_{n}$ denote the immediate
successor of~$M_{n -1}$ in $\mathfrak{N}$. Since the order type of
$\mathfrak{N}$ is $\omega,$ $\bigcup_{n =1}^{\infty}M_{n} =\mathfrak{X}$.

Suppose that~$T \in \mathcal{A}_{ -}$ and that~$\varphi \in \mathcal{M}_{n}^{
\bot} .$ Then~$T\varphi \in \mathfrak{N}^{ \bot}(\varphi)_{ -} \subset
\mathfrak{N}^{ \bot}(\varphi) \subseteq \mathcal{M}_{n}^{ \bot}$. Since
$\mathcal{M}_{n +1}^{\perp}$ is the immediate predecessor of $\mathcal{M}%
_{n}^{\perp}$in $\mathfrak{N}^{\perp}$, it follows that $T\varphi
\in \mathcal{M}_{n +1}$, and so $T(\mathcal{M}_{n}^{ \bot}) \subseteq
\mathcal{M}_{n +1}^{\perp} .$

Suppose that~$A \in \mathcal{A}$. Then $AT \in \mathcal{A}_{ -}$ and so
$AT(M_{n}^{\perp}) \subseteq M_{n +1}^{\perp}$ for each $n \geq0$ and so
$(AT)^{n}(\mathfrak{X}^{ \prime}) =(AT)^{n}(M_{0}^{\perp}) \subseteq
M_{n}^{\perp}$ for each $n \geq0$.

Let $S =$~ $1 +\sum_{n =1}^{\infty}(AT)^{n}$. The sum $S$ is well-defined as
an operator in $\mathcal{L}(\mathfrak{X}^{ \prime})$ because, for each $x
\in \mathfrak{X}$ and each $\varphi \in \mathfrak{X}^{ \prime}$, the
series~$\sum_{n =1}^{\infty}(AT)^{n})(\varphi)(x)$ has only finitely many
non-zero terms. (To see this note that $x \in M_{n^{\#}}$ for some $n^{\#}
\geq0$, and $((AT)^{n}\varphi)(x) =0$ if~\thinspace$n \geq n^{\#} .$)
Furthermore $S(1 -AT)\varphi(x) =(1 -AT)S\varphi(x) =\varphi(x)$, and so~$S
=(1 -AT)^{ -1}$. Finally, it is easy to check that $S(M_{n}^{\perp}) \subseteq
M_{n}^{\perp}$ for each $n \geq0$ and so $S \in \mathcal{A}$. So~$T
\in \operatorname*{Rad} \mathcal{A}$, and hence $\mathcal{A}_{ -}
\subseteq \operatorname*{Rad} \mathcal{A}$. It follows from~ Lemma~\ref{Rad1}
that~~$\mathcal{A}_{ -} =\operatorname*{Rad} \mathcal{A}$.
\end{proof}
\end{theorem}

\subsection{An example}

The nest~$\mathfrak{N}$ defined in Example \ref{c00(N)} satisfies the
conditions of Theorem~\ref{Donthmmod}, and so~$(\operatorname*{Alg}
\mathfrak{N}^{ \bot})_{ -} =\operatorname*{Rad} (\operatorname*{Alg}
\mathfrak{N}^{ \bot}) .$ Note that $\mathfrak{N}^{\perp}$ is not well-ordered.
It does, however, satisfy the ascending chain condition, i.e. each subset of
$\mathfrak{N}^{\perp}$ contains a maximal element.

\begin{definition}
Suppose that~$\mathfrak{X}_{1}$ and $\mathfrak{X}_{2}$ are vector spaces over
the same field~$\mathbb{F}$, and that $\mathfrak{N}_{k}$ is a nest of
subspaces of $\mathfrak{X}_{k}$ for~$k \in \{1 ,2\}$. The ordinal sum
$\mathfrak{N}_{1} \dotplus \mathfrak{N}_{2}$ is a nest of subspaces of
$\mathfrak{X} =\mathfrak{X}_{1} \oplus \mathfrak{X}_{2}$ defined by%
\[
\mathfrak{N}_{1} \dotplus \mathfrak{N}_{2} =\{N \oplus \{0\} :N \in
\mathfrak{N}_{1}\} \cup \{ \mathfrak{X}_{1} \oplus N :N \in \mathfrak{N}_{2}\}
\]

\end{definition}

Let $\mathcal{A} =\operatorname*{Alg} (\mathfrak{N}_{1} \dotplus
\mathfrak{N}_{2})$ and let $\mathcal{A}_{k} =\operatorname*{Alg}
\mathfrak{N}_{k}$ for~$k \in \{1 ,2\} .$~Every $T$ in $\mathcal{L} \left(
X\right) $ has an operator matrix,
\[
T =\left(
\begin{array}
[c]{cc}%
A_{1} & B\\
C & A_{2}%
\end{array}
\right)
\]
relative to the decomposition $\mathfrak{X} =\mathfrak{X}_{1} \oplus
\mathfrak{X}_{2}$. It is easy to check that%
\begin{equation}
T \in \mathcal{A}\  \text{if and only if}\ A_{k}\text{ } \in \mathcal{A}%
_{k}\  \text{for}\ k \in \{1 ,2\} \  \text{and}\ C =0 ,\  \text{and}\label{Ordsum1}%
\end{equation}%
\begin{equation}
T \in \mathcal{A}_{ -}\  \text{if and only if}\ A_{k}\text{ } \in(\mathcal{A}%
_{k})_{ -}\  \text{for}\ k \in \{1 ,2\} \  \text{and}\ C =0.\label{Ordsum2}%
\end{equation}

\begin{lemma}
\label{ordsumrad}With the above notation and~$C =0 ,$
\begin{equation}
T \in \operatorname*{Rad} \mathcal{A}\  \text{if and only if}\ A_{k}\text{ }
\in \operatorname*{Rad} \mathcal{A}_{k}\  \text{for}\ k \in1 ,2\} ,\  \text{and}%
\label{Ordsum3}%
\end{equation}%
\begin{equation}
\operatorname*{Rad} \mathcal{A} =\mathcal{A}_{ -}\  \text{if and only
if}\  \operatorname*{Rad} \text{ }\mathcal{A}_{k} =(\mathcal{A}_{k})_{
-}\  \text{for}\ k \in \{1 ,2\} \label{Ordsum4}%
\end{equation}

\begin{proof}
A simple matrix computation shows that if~$\left(
\begin{array}
[c]{cc}%
D & E\\
0 & F
\end{array}
\right)  =\left(
\begin{array}
[c]{cc}%
A_{1} & B\\
0 & A_{2}%
\end{array}
\right) ^{ -1}$ if and only if~$D =A_{1}^{ -1}$, $F =A_{2}^{ -1}$ and $E =
-A_{1}^{ -1}BA_{2}^{ -1} .$ So $\left(
\begin{array}
[c]{cc}%
A_{1} & B\\
0 & A_{2}%
\end{array}
\right) ^{ -1} \in \mathcal{A}$ if and only if~$A_{1}^{ -1} \in \mathcal{A}%
_{1}\ $ and $A_{2}^{ -1~} \in \mathcal{A}_{2}$. Statement (\ref{Ordsum3}) is
now obvious. Statement (\ref{Ordsum4}) follows from (\ref{Ordsum2})and
(\ref{Ordsum3}).
\end{proof}
\end{lemma}

\begin{example}
Let $\mathfrak{X} =\mathfrak{Y} \oplus \mathfrak{Y}$, where $\mathfrak{Y}$ is
the vector space of all~$\mathbb{F}$-valued sequences. Let
\[
\mathfrak{N}_{1} =\{ \mathfrak{Y} ,M_{1}^{ \bot} ,M_{2}^{ \bot} ,M_{3}^{ \bot}
, \cdots,\{0\} \} \  \  \
\]
where~$M_{n}^{ \bot} =\{ \varphi \in \mathfrak{Y} :\  \operatorname*{supp}
\varphi \subseteq \{n +1 ,n +2 ,n +3 , \cdots \} ,$as in Example \ref{c00(N)},
and let~
\[
\mathfrak{N}_{2} =\{ \{0\} ,(M_{1}^{\#})^{ \bot} ,(M_{2}^{\#})^{ \bot}
,(M_{3}^{\#})^{ \bot} , \cdots, ,\mathfrak{Y}\}
\]
where $(M_{n}^{\#})^{ \bot} =\{ \varphi \in \mathfrak{Y} :\  \operatorname*{supp}
\varphi \subseteq \{1 ,2 , \cdots,n\} \} ,$ as in Example \ref{c00(N)1}.

Note that~$\mathfrak{N}_{1} =\mathfrak{N}^{ \bot}$, where $\mathfrak{N}$ is as
defined in~Example \ref{c00(N)}. Since~$\mathfrak{N}_{1}$ is well-ordered with
order type $\omega$, it follows from Theorem \ref{dualrad} that
$\operatorname*{Rad} \mathcal{A}_{1} =(\mathcal{A}_{1})_{ -} .$ Note also that
$\mathfrak{N}_{2}$ is well-ordered, i.e. it satisfies the descending chain
condition.~So by Corollary \ref{wellordcor} $\operatorname*{Rad}
\mathcal{A}_{2} =(\mathcal{A}_{2})_{ -}$. So by Lemma \ref{ordsumrad}
$\operatorname*{Rad} (\mathfrak{N}_{1} \dotplus \mathfrak{N}_{2})
=(\mathfrak{N}_{1} \dotplus \mathfrak{N}_{2})_{ -}$.

But $\mathfrak{N}_{1} \dotplus \mathfrak{N}_{2})$~satisfies neither the
ascending chain condition nor the ascending chain condition. Its order type
is~$1 +\omega^{ \ast} +\omega+1$, i.e. the order type of $\{ -\infty \}
\cup \mathbb{Z} \cup \{ \infty \}$, where $\mathbb{Z}$ denote the set of integers,
and it contains both strictly decreasing and strictly increasing infinite
sequences of subspaces.
\end{example}

\end{document}